\documentclass[11pt]{article}
\usepackage{amsfonts}
\usepackage{mathrsfs}
\usepackage[dvips]{graphics}
\usepackage[dvips]{color}
\usepackage{amsthm}
\usepackage[]{amsmath}
\usepackage{CJK}
\usepackage{indentfirst}
\newtheorem{proposition}{Proposition}[section]
\newtheorem{lemma}[proposition]{Lemma}

\newtheorem{theorem}[proposition]{Theorem}
\newtheorem{corollary}[proposition]{Corollary}
\newtheorem{property}[proposition]{Property}

\begin{document}
\begin{CJK*}{GBK}{song}
\CJKindent

\centerline{\textbf{\LARGE{The numbers of repeated palindromes}}}

\vspace{0.2cm}

\centerline{\textbf{\LARGE{in the Fibonacci and Tribonacci sequences}}}

\vspace{0.2cm}

\centerline{Huang Yuke\footnote[1]{School of Mathematics and Systems Science, Beihang University (BUAA), Beijing, 100191, P. R. China. E-mail address: huangyuke07@tsinghua.org.cn,~hyg03ster@163.com (Corresponding author).}
~~Wen Zhiying\footnote[2]{Department of Mathematical Sciences, Tsinghua University, Beijing, 100084, P. R. China. E-mail address: wenzy@tsinghua.edu.cn.}}

\vspace{1cm}

\centerline{\textbf{\large{ABSTRACT}}}

\vspace{0.2cm}

The Fibonacci sequence $\mathbb{F}$ is the fixed point beginning with $a$ of morphism $\sigma(a,b)=(ab,a)$.
Since $\mathbb{F}$ is uniformly recurrent, each factor $\omega$ appears infinite many times in the sequence which is arranged as $\omega_p$ $(p\ge 1)$. Here we distinguish $\omega_p\neq\omega_q$ if $p\neq q$. In this paper, we give algorithm for counting the number of repeated palindromes in $\mathbb{F}[1,n]$ (the prefix of $\mathbb{F}$ of length $n$). That is the number of the pairs $(\omega, p)$, where $\omega$ is a palindrome and $\omega_p\prec\mathbb{F}[1,n]$.
We also get explicit expressions for some special $n$ such as $n=f_m$ (the $m$-th Fibonacci number).
The similar results are also given to the Tribonacci sequence, the fixed point beginning with $a$ of morphism $\tau(a,b,c)=(ab,ac,a)$.

\vspace{0.2cm}


\vspace{0.2cm}

\noindent\textbf{2010 MR Subject Classification}~~~~11B85; 68Q45

\section{Introduction}

A palindrome is a finite word that reads the same backwards as forwards.
Let $\mathcal{P}_F$ (resp. $\mathcal{P}_T$) be all palindromes occurring in $\mathbb{F}$ (resp. $\mathbb{T}$).
The research on the number of distinct palindromes is called ``rich word".
A finite word $\omega$ is rich if and only if $\omega$ contains exactly $|\omega|+1$ distinct palindromes (including the empty word), where $|\omega|$ is the length of $\omega$. A sequence is rich if and only if all of its factors are rich.
X.Droubay, J.Justin and G.Pirillo\cite{DJP2001} proved that episturmian sequences are rich. Therefore, as special cases, $\mathbb{F}$ and $\mathbb{T}$ are rich. And so the number of distinct palindromes in
$\mathbb{F}[1,n]$ (resp. $\mathbb{T}[1,n]$, the prefix of $\mathbb{T}$ of length $n$) is $n+1$ for all $n$.

In this paper, we consider the numbers of repeated palindromes in $\mathbb{F}[1,n]$ and $\mathbb{T}[1,n]$. Denote
$$A(n)=\sharp\{(\omega,p):\omega\in\mathcal{P}_F,\omega_p\prec\mathbb{F}[1,n]\}
\text{ and }B(n)=\sharp\{(\omega,p):\omega\in\mathcal{P}_T,\omega_p\prec\mathbb{T}[1,n]\}.$$
The research on counting the repeated palindromes is not rich.
From our knowledge, it seems the first time to study this problem.
In related fields, the numbers of repeated squares and cubes in $\mathbb{F}[1,n]$ and $\mathbb{T}[1,n]$
are given for some special cases.
Let $f_m$ be the $m$-th Fibonacci number.
$f_{-2}=0$, $f_{-1}=1$, $f_{m}=f_{m-1}+f_{m-2}$ for $m\geq 0$.
Let $t_m$ be the $m$-th Tribonacci number.
$t_{-2}=0$, $t_{-1}=t_{0}=1$, $t_m=t_{m-1}+t_{m-2}+t_{m-3}$ for $m\geq1$.
A.S.Fraenkel and J.Simpson\cite{FS2014} gave the number of repeated squares in $\mathbb{F}[1,f_m]$.
In 2014,
C.F.Du, H.Mousavi, L.Schaeffer, J.Shallit\cite{DMSS2014} and H.Mousavi, J.Shallit\cite{MS2014} gave the numbers of repeated squares and cubes in $\mathbb{F}[1,f_m]$ and $\mathbb{T}[1,t_m]$ by mechanical methods.

The main difficulty of this problem is twofold:
(1) The positions of all occurrences for all palindromes are not easy to be determined.
In this paper, we overcome this difficulty by using the ``gap sequence" properties of $\mathbb{F}$ and $\mathbb{T}$, which we introduced and studied in \cite{HW2015-1,HW2015-2}.
(2) Taking $\mathbb{F}$ for instance, by the gap sequence property of $\mathbb{F}$, we can find out all distinct palindromes in $\mathbb{F}[1,n]$. We can also count the number of occurrences of each palindrome. So the summation of these numbers are the numbers of repeated palindromes in $\mathbb{F}[1,n]$. But this method is complicated.
We overcome this difficulty by studying the relations among positions of each $\omega_p$, and establishing the recursive structure of $\mathcal{P}_F$ in Section 3. The similar structure of $\mathcal{P}_T$ is established in Section 6.

Using the gap sequence properties and recursive structures, we give algorithms for counting $A(n)$ and $B(n)$ in Section 4 and 7, respectively.
We also get explicit expressions for some special $n$, such as:
for $m\geq0$, $A(f_m)=\tfrac{m-3}{5}f_{m+2}+\tfrac{m-1}{5}f_{m}+m+3$ and
$$B(t_m)=\tfrac{m}{22}\big(10t_{m}+5t_{m-1}+3t_{m-2}\big)
+\tfrac{1}{22}\big(-23t_{m}+12t_{m-1}-5t_{m-2}\big)+m+\tfrac{3}{2}.$$

We think this method for counting the repeated palindromes is fit for $m$-bonacci sequence,
and even fit for sturmian sequence, episturmian sequence etc.
But now we only have the gap sequence properties of $\mathbb{F}$ and $\mathbb{T}$.
As a final remark, we establish the cylinder structures and chain structures of $\mathcal{P}_F$ and $\mathcal{P}_T$ in Section 8. Using them, we prove some known results.

\vspace{0.2cm}

Let $\rho=x_1\cdots x_n$ be a finite word (or $\rho=x_1x_2\cdots$ be a sequence).
For any $i\leq j\leq n$, we define $\rho[i,j]=x_ix_{i+1}\cdots x_{j-1}x_j$.
For convenience, we denote $\rho[i]=\rho[i,i]=x_i$, $\rho[i,i-1]=\varepsilon$ (empty word).

We say that $\nu$ is a prefix (resp. suffix) of a word $\omega$ if there exists word $u$ such that $\omega=\nu u$ (resp. $\omega=u\nu$), $|u|\geq0$, which denoted by $\nu\triangleleft\omega$ (resp. $\nu\triangleright\omega$).
In this case, we write $\nu^{-1}\omega=u$ (resp. $\omega\nu^{-1}=u$), where $\nu^{-1}$ is the inverse word of $\nu$ such that $\nu\nu^{-1}=\nu^{-1}\nu=\varepsilon$.

Let $\omega$ be a factor of $\mathbb{F}$, denoted by $\omega\prec\mathbb{F}$. Since $\mathbb{F}$ is uniformly recurrent, i.e., each factor $\omega$ occurs infinitely often and with bounded gaps between consecutive occurrences\cite{AS2003}.
We arrange them in the sequence $\{\omega_p\}_{p\geq1}$, where
$\omega_p$ denote the $p$-th occurrence of $\omega$.
We denote by $L(\omega,p)$ and $P(\omega,p)$ the positions of the first and last letters of $\omega_p$, respectively. Then $P(\omega,p)=L(\omega,p)+|\omega|-1$.
We denote the gap between $\omega_p$ and $\omega_{p+1}$ by $G_p(\omega)$.
Specifically, for $p\geq1$, let $\omega_p=x_{i+1}\cdots x_{i+n}$ and $\omega_{p+1}=x_{j+1}\cdots x_{j+n}$.
Then when $i+n<j$,
$G_p(\omega)=x_{i+n+1}\cdots x_{j}$;
when $i+n=j$, $G_p(\omega)=\varepsilon$; when $i+n>j$, $G_p(\omega)=(x_{j+1}\cdots x_{i+n})^{-1}$.
The sequence $\{G_p(\omega)\}_{p\geq1}$ is called the gap sequence of factor $\omega$.

\section{Preliminaries of the Fibonacci sequence}

We define $F_m=\sigma^m(a)$ for $m\geq0$, $F_{-2}=\varepsilon$, $F_{-1}=b$.
Then $F_0=a$, $F_m=F_{m-1}F_{m-2}$ for $m\geq1$, and $|F_m|=f_m$ for $m\geq-2$.
Denote by $\delta_m$ the last letter of $F_m$ for $m\geq-1$, then $\delta_m=a$ if and only if $m$ is even.
The $m$-th kernel word of $\mathbb{F}$ is defined as $K_m=\delta_{m+1}F_m\delta_m^{-1}$, which is also called singular word.
It is known that all kernel words are palindromes and $K_m=K_{m-2}K_{m-3}K_{m-2}$ for all $m\geq2$, see \cite{WW1994}.
Let $Ker(\omega)$ be the maximal kernel word occurring in factor $\omega$. Then by Theorem 1.9 in \cite{HW2015-1}, $Ker(\omega)$ occurs in $\omega$ only once.
Moreover

\begin{property}[Theorem 2.8 in \cite{HW2015-1}]\label{wpf}
$Ker(\omega_p)=Ker(\omega)_p$ for all $\omega\in\mathbb{F}$ and $p\geq1.$
\end{property}
This means, let $Ker(\omega)=K_m$, then the maximal kernel word occurring in $\omega_p$ is just $K_{m,p}$. For instance, $Ker(aba)=b$, $(aba)_3=\mathbb{F}[6,8]$, $(b)_3=\mathbb{F}[7]$, so $Ker((aba)_3)=(b)_3$, $(aba)_3=a(b)_3a$.

In \cite{HW2015-1}, we use an equivalent notion of ``gap'' called ``return word'', which is introduced by F.Durand\cite{D1998}. Let $r_p(\omega)$ be the $p$-th return word of $\omega$, then $r_p(\omega)=\omega G_p(\omega)$. Using this relation, we rewrite Theorem 2.11 in \cite{HW2015-1} as below.

\begin{property}[]\label{Gf}
The gap sequence $\{G_p(\omega)\}_{p\geq1}$ is $\mathbb{F}$ over the alphabet $\{G_1(\omega),G_2(\omega)\}$ for all $\omega\in\mathbb{F}$.
\end{property}

\section{The recursive structure of $\mathcal{P}_F$}

A known result is $|\mathbb{F}[1,p-1]|_a=\lfloor\phi p\rfloor$, where
$\phi=\frac{\sqrt{5}-1}{2}$ and $\lfloor\alpha\rfloor$ is the biggest integer not larger than $\alpha$.
Since $P(\omega,p)=L(\omega,p)+|\omega|-1$, we rewrite Proposition 3.4 in \cite{HW2015-1} as below.

\begin{property}\label{Pf}
$P(K_m,p)=pf_{m+1}+(\lfloor\phi p\rfloor+1)f_{m}-1$ for $m\geq-1$, $p\geq1$.
\end{property}

\noindent\emph{Example.} $P(a,p)=p+\lfloor\phi p\rfloor$,
$P(b,p)=2p+\lfloor\phi p\rfloor$
and $P(aa,p)=3p+2\lfloor\phi p\rfloor+1$.

\vspace{0.2cm}

By Corollary 2.10 in \cite{HW2015-1} and $K_{m+3}=K_{m+1}K_m K_{m+1}$, any factor $\omega$ with kernel $K_m$ can be expressed uniquely as
$\omega=K_{m+1}[i,f_{m+1}] K_m K_{m+1}[1,j]=K_{m+3}[i,f_{m+2}+j],$
where $2\leq i\leq f_{m+1}+1$, $0\leq j\leq f_{m+1}-1$, $m\geq-1$.
Let $\omega\in \mathcal{P}_F$. Since both $\omega$ and $Ker(\omega)$ are palindromes, and
$Ker(\omega)$ occurs only once in $\omega$, $Ker(\omega)$ occurs in the middle of $\omega$.
Thus

\begin{property}[]\
Any palindrome with kernel $K_m$ can be expressed uniquely as
$$K_{m+1}[i+1,f_{m+1}] K_m K_{m+1}[1,f_{m+1}-i]=K_{m+3}[i+1,f_{m+3}-i],\eqno(1)$$
where $1\leq i\leq f_{m+1}$ and $m\geq-1$.
\end{property}

\begin{property}[]\
Let $\omega$ be a palindrome with kernel $K_m$ satisfying the expression (1), then
$$P(\omega,p)=P(K_m,p)+f_{m+1}-i=pf_{m+1}+(\lfloor\phi p\rfloor+1)f_{m}-1+f_{m+1}-i,\eqno(2)$$
where $1\leq i\leq f_{m+1}$, $m\geq-1$ and $p\geq1$.
\end{property}

\noindent\emph{Remark.} Obviously $\{n:\mathbb{F}[1,n]\in\mathcal{P}_F\}
=\{n:\omega\in\mathcal{P}_F,|\omega|=P(\omega,1)=n\}$.
By expressions (1) and (2), $|\omega|=f_{m+3}-2i$, $P(\omega,1)=2f_{m+1}+f_m-i-1$.
Thus $|\omega|=P(\omega,1)$ if and only if $i=f_{m+3}-2f_{m+1}-f_m+1=1\in\{1,\cdots,f_{m+1}\}.$
This means $\mathbb{F}[1,n]$ is a palindrome if and only if $n=f_{m+3}-2$ for $m\geq-1$,
i.e. $n=f_m-2$ for $m\geq2$, which is Theorem 14 in \cite{DMSS2014}.

\vspace{0.2cm}

Define
$\langle K_m,p\rangle=\{P(\omega,p):\omega\in\mathcal{P}_F,Ker(\omega)=K_m\}$ for $m\geq-1$, $p\geq 1$,
which contains some consecutive integers.
By the expression (2),
\setcounter{equation}{2}
\begin{equation}
\begin{split}
\langle K_m,p\rangle=&\{P(K_m,p)+f_{m+1}-i,1\leq i\leq f_{m+1}\}
=\{P(K_m,p),\cdots,P(K_m,p)+f_{m+1}-1\}\\
=&\{pf_{m+1}+(\lfloor\phi p\rfloor+1) f_{m}-1,\cdots,(p+1)f_{m+1}+(\lfloor\phi p\rfloor+1)f_{m}-2\}.
\end{split}
\end{equation}
An immediately corollary is $\sharp\langle K_m,p\rangle=\sharp\{1\leq i\leq f_{m+1}\}=f_{m+1}$.

\begin{lemma}\label{Lf}
$\lfloor\phi(p+\lfloor\phi p\rfloor+1)\rfloor=p$ and $\lfloor\phi(2p+\lfloor\phi p\rfloor+1)\rfloor=p+\lfloor\phi p\rfloor$.
\end{lemma}

\begin{proof}
(1) By Property \ref{Gf}, the $p$-th occurrence of $aba=aG_1(a)a$ is equivalent to the $P(a,p)$-th occurrence of $a$.
Moreover, the $(P(a,p)+1)$-th letter $a$ occurs at position $P(aba,p)$.
So $P(aba,p)=P(a,P(a,p)+1)$.
Since $Ker(aba)=b$, by expression (2), $P(aba,p)=P(b,p)+1=2p+\lfloor\phi p\rfloor+1$.

Since $P(a,p)=p+\lfloor\phi p\rfloor$, $P(a,P(a,p)+1)=P(a,p+\lfloor\phi p\rfloor+1)=p+\lfloor\phi p\rfloor+1+\lfloor\phi (p+\lfloor\phi p\rfloor+1)\rfloor$.

Comparing the two expressions, we have $\lfloor\phi(p+\lfloor\phi p\rfloor+1)\rfloor=p$.

(2) Similarly, since $aa=aG_2(a)a$, the $(P(b,p)+1)$-th letter $a$ occurs at position $P(aa,p)$. This means $P(aa,p)=P(a,P(b,p)+1)$.
By Property \ref{Pf}, the second equation holds.
\end{proof}

\noindent\emph{Remark.} By analogous arguments, using $P(aba,p)-2=P(a,P(a,p))$ and $P(aa,p)-1=P(a,P(b,p))$, we can get $\lfloor\phi (p+\lfloor\phi p\rfloor)\rfloor=p-1$ and $\lfloor\phi (2p+\lfloor\phi p\rfloor)\rfloor=p+\lfloor\phi p\rfloor$.

\begin{property}\label{P4.1}
$\langle K_m,p\rangle=\langle K_{m-2},P(b,p)+1\rangle\cup\langle K_{m-1},P(a,p)+1\rangle$ for $m,p\geq1$.
\end{property}

\begin{proof} We can prove this property by comparing the minimal and maximal elements in these sets. For instance,
$\min\langle K_m,p\rangle=pf_{m+1}+(\lfloor\phi p\rfloor+1) f_{m}-1$. By Lemma \ref{Lf},
\begin{equation*}
\begin{split}
&\min\langle K_{m-2},P(b,p)+1\rangle
=(P(b,p)+1)f_{m-1}+(\lfloor\phi (P(b,p)+1)\rfloor+1) f_{m-2}-1\\
=&(2p+\lfloor\phi p\rfloor+1)f_{m-1}+(\lfloor\phi (2p+\lfloor\phi p\rfloor+1)\rfloor+1) f_{m-2}-1\\
=&(2p+\lfloor\phi p\rfloor+1)f_{m-1}+(p+\lfloor\phi p\rfloor+1) f_{m-2}-1=\min\langle K_m,p\rangle.
\end{split}
\end{equation*}
Similarly, $\max\langle K_{m-2},P(b,p)+1\rangle+1=\min\langle K_{m-1},P(a,p)+1\rangle$, $\max\langle K_m,p\rangle=\max\langle K_{m-1},P(a,p)+1\rangle$.
Thus the conclusion holds.
\end{proof}

\noindent\emph{Example.} Take $m=4$ and $p=1$, then $P(b,p)+1=3$ and $P(a,p)+1=2$. Thus $\langle K_4,1\rangle=\{20,\cdots,32\}$ is the disjoint unite of $\langle K_2,3\rangle=\{20,\cdots,24\}$ and $\langle K_3,2\rangle=\{25,\cdots,32\}$.

\begin{property}[]\label{P4.2}
$\max\langle K_0,p\rangle=\max\langle b,p\rangle=\langle a,P(a,p)+1\rangle$  for $p\geq1$.
\end{property}

By Property \ref{P4.1} and \ref{P4.2}, we establish the following recursive relations for any $\langle K_m,p\rangle$. We call it the recursive structure of $\mathcal{P}_F$.
\begin{equation*}
\begin{cases}
\pi\langle K_m,p\rangle=\langle K_{m-2},P(b,p)+1\rangle\cup\langle K_{m-1},P(a,p)+1\rangle&\text{for }m\geq1;\\
\pi\langle K_0,p\rangle=\pi\langle b,p\rangle=\langle K_{-1},P(a,p)+1\rangle.
\end{cases}
\end{equation*}

On the other hand, (a) for any $m\geq-1$, each $\langle K_{m},1\rangle$ belongs to the recursive structure.
(b) Since $\mathbb{F}$ over alphabet $\{a,b\}$, $\mathbb{N}=\{1\}\cup\{P(a,\hat{p})+1,\hat{p}\geq1\}\cup\{P(b,\hat{p})+1,\hat{p}\geq1\}$. So for any $p\geq2$, there exists
$\hat{p}\geq1$ such that $P(a,\hat{p})+1=p$ or $P(b,\hat{p})+1=p$.
\begin{equation*}
\begin{cases}
\text{If there exists }\hat{p}\text{ such that }P(a,\hat{p})+1=p,
\begin{cases}
\langle K_{m},p\rangle\subset\pi\langle K_{m+1},\hat{p}\rangle\text{ for }m\geq0\\
\langle K_{m},p\rangle=\langle a,p\rangle=\pi\langle K_{m+1},\hat{p}\rangle\text{ for }m=-1
\end{cases}\\
\text{If there exists }\hat{p}\text{ such that }P(b,\hat{p})+1=p,
\langle K_{m},p\rangle\subset\pi\langle K_{m+2},\hat{p}\rangle\text{ for }m\geq-1.
\end{cases}
\end{equation*}
Thus the recursive structure contains all $\langle K_m,p\rangle$, i.e. contains all palindromes in $\mathbb{F}$.
\vspace{-0.3cm}
\small
\setlength{\unitlength}{1.2mm}
\begin{center}
\begin{picture}(125,65)
\put(0,0.5){\line(1,0){11}}
\put(0,63.5){\line(1,0){11}}
\put(0,0.5){\line(0,1){63}}
\put(11,0.5){\line(0,1){63}}
\put(4,1){32}
\put(4,4){31}
\put(4,10){30}
\put(4,16){29}
\put(4,19){28}
\put(4,25){27}
\put(4,28){26}
\put(4,34){25}
\put(4,40){24}
\put(4,43){23}
\put(4,49){22}
\put(4,55){21}
\put(4,58){20}
\put(1,61){$\langle K_4,1\rangle$}
\put(21,0.5){\line(1,0){11}}
\put(21,39.5){\line(1,0){11}}
\put(21,0.5){\line(0,1){39}}
\put(32,0.5){\line(0,1){39}}
\put(25,1){32}
\put(25,4){31}
\put(25,10){30}
\put(25,16){29}
\put(25,19){28}
\put(25,25){27}
\put(25,28){26}
\put(25,34){25}
\put(22,37){$\langle K_3,2\rangle$}
\put(42,39.5){\line(1,0){11}}
\put(42,63.5){\line(1,0){11}}
\put(42,39.5){\line(0,1){24}}
\put(53,39.5){\line(0,1){24}}
\put(46,40){24}
\put(46,43){23}
\put(46,49){22}
\put(46,55){21}
\put(46,58){20}
\put(43,61){$\langle K_2,3\rangle$}
\put(42,0.5){\line(1,0){11}}
\put(42,24.5){\line(1,0){11}}
\put(42,0.5){\line(0,1){24}}
\put(53,0.5){\line(0,1){24}}
\put(46,1){32}
\put(46,4){31}
\put(46,10){30}
\put(46,16){29}
\put(46,19){28}
\put(43,22){$\langle K_2,4\rangle$}
\put(63,0.5){\line(0,1){15}}
\put(74,0.5){\line(0,1){15}}
\put(63,0.5){\line(1,0){11}}
\put(63,15.5){\line(1,0){11}}
\put(67,40){24}
\put(67,43){23}
\put(67,49){22}
\put(64,52){$\langle K_1,5\rangle$}
\put(63,24.5){\line(0,1){30}}
\put(74,24.5){\line(0,1){30}}
\put(63,24.5){\line(1,0){11}}
\put(63,39.5){\line(1,0){11}}
\put(63,54.5){\line(1,0){11}}
\put(67,25){27}
\put(67,28){26}
\put(67,34){25}
\put(64,37){$\langle K_1,6\rangle$}
\put(67,1){32}
\put(67,4){31}
\put(67,10){30}
\put(64,13){$\langle K_1,7\rangle$}
\put(84,54.5){\line(0,1){9}}
\put(96,54.5){\line(0,1){9}}
\put(84,54.5){\line(1,0){12}}
\put(84,63.5){\line(1,0){12}}
\put(88,55){21}
\put(88,58){20}
\put(85,61){$\langle K_0,8\rangle$}
\put(84,39.5){\line(0,1){9}}
\put(96,39.5){\line(0,1){9}}
\put(84,48.5){\line(1,0){12}}
\put(84,39.5){\line(1,0){12}}
\put(88,40){24}
\put(88,43){23}
\put(85,46){$\langle K_0,9\rangle$}
\put(84,15.5){\line(0,1){18}}
\put(96,15.5){\line(0,1){18}}
\put(84,15.5){\line(1,0){12}}
\put(84,24.5){\line(1,0){12}}
\put(84,33.5){\line(1,0){12}}
\put(88,25){27}
\put(88,28){26}
\put(85,31){$\langle K_0,10\rangle$}
\put(88,16){29}
\put(88,19){28}
\put(85,22){$\langle K_0,11\rangle$}
\put(84,0.5){\line(0,1){9}}
\put(96,0.5){\line(0,1){9}}
\put(84,9.5){\line(1,0){12}}
\put(84,0.5){\line(1,0){12}}
\put(88,1){32}
\put(88,4){31}
\put(85,7){$\langle K_0,12\rangle$}
\put(106,0.5){\line(0,1){63}}
\put(120,0.5){\line(0,1){63}}
\put(106,0.5){\line(1,0){14}}
\put(106,9.5){\line(1,0){14}}
\put(106,15.5){\line(1,0){14}}
\put(106,24.5){\line(1,0){14}}
\put(106,33.5){\line(1,0){14}}
\put(106,39.5){\line(1,0){14}}
\put(106,48.5){\line(1,0){14}}
\put(106,54.5){\line(1,0){14}}
\put(106,63.5){\line(1,0){14}}
\put(112,55){21}
\put(107,61){$\langle K_{-1},13\rangle$}
\put(112,49){22}
\put(107,52){$\langle K_{-1},14\rangle$}
\put(112,40){24}
\put(107,46){$\langle K_{-1},15\rangle$}
\put(112,34){25}
\put(107,37){$\langle K_{-1},16\rangle$}
\put(112,25){27}
\put(107,31){$\langle K_{-1},17\rangle$}
\put(112,16){29}
\put(107,22){$\langle K_{-1},18\rangle$}
\put(112,10){30}
\put(107,13){$\langle K_{-1},19\rangle$}
\put(112,1){32}
\put(107,7){$\langle K_{-1},20\rangle$}
\put(12,50){\vector(1,0){29}}
\put(12,50){\vector(1,-2){8}}
\put(33,32){\vector(1,0){29}}
\put(33,32){\vector(1,-2){8}}
\put(54,60){\vector(1,0){29}}
\put(54,60){\vector(1,-2){8}}
\put(54,20){\vector(1,0){29}}
\put(54,20){\vector(1,-2){8}}
\put(75,52){\vector(1,0){30}}
\put(75,52){\vector(1,-1){8}}
\put(75,36){\vector(1,0){30}}
\put(75,36){\vector(1,-1){8}}
\put(75,12){\vector(1,0){30}}
\put(75,12){\vector(1,-1){8}}
\put(97,5){\vector(1,0){8}}
\put(97,20){\vector(1,0){8}}
\put(97,29){\vector(1,0){8}}
\put(97,44){\vector(1,0){8}}
\put(97,59){\vector(1,0){8}}
\end{picture}
\end{center}
\normalsize
\vspace{-0.2cm}
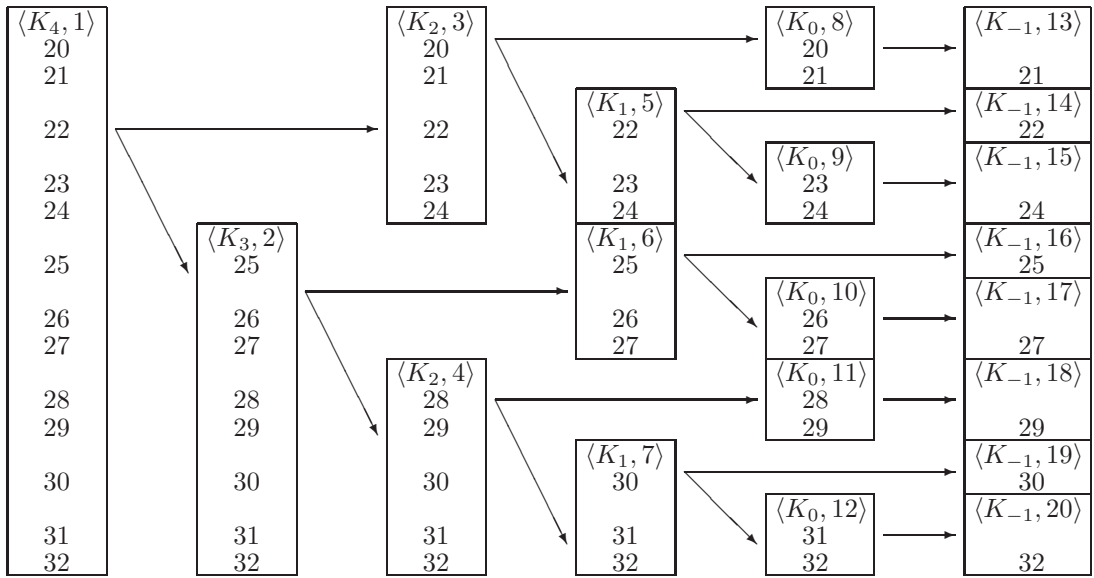
\centerline{Fig.1: The recursive structure of $\mathcal{P}_F$ from $\langle K_4,1\rangle$.}

\vspace{0.2cm}

By the recursive structure, we have the relation between the number of palindromes ending at position $pf_{m+1}+(\lfloor\phi p\rfloor+1) f_{m}+i-2$ and $f_{m+2}+i-2$, $1\leq i\leq f_{m+1}$. They are the $i$-th element in $\langle K_m,p\rangle$ and $\langle K_m,1\rangle$, respectively.

\begin{property}[]\label{P5.5} For $1\leq i\leq f_{m+1}$, $m,p\geq1$,
\begin{equation*}
\begin{split}
&\{\omega:\omega\in\mathcal{P}_F,\omega\triangleright\mathbb{F}[1,f_{m+2}+i-2]\}\\
=&\{\omega:\omega\in\mathcal{P}_F,\omega\triangleright\mathbb{F}[1,pf_{m+1}+(\lfloor\phi p\rfloor+1) f_{m}+i-2],Ker(\omega)=K_j,-1\leq j\leq m\}.
\end{split}
\end{equation*}
\end{property}

\noindent\emph{Example.} Taking $m=2$, $p=3$, $i=2$. All palindromes ending at position 8 are $\{a,aba,ababa\}$.
All palindromes ending at position 21 are $\{a,aba,ababa,\omega\}$ where $\omega=ababaababa$.
Since $Ker(a)=K_{-1}$, $Ker(aba)=K_0$, $Ker(ababa)=K_2$ and $Ker(\omega)=K_4$,
only $\{a,aba,ababa\}$ are palindromes with kernel $K_{j}$, $-1\leq j\leq 2$.

\section{The algorithm for counting $A(n)$}

By the recursive structure of $\mathcal{P}_F$, we can count the number of palindromes ending at position $n$. We denote by $a(n)=\sharp\{(\omega,p):\omega\in\mathcal{P}_F,\omega_p\triangleright\mathbb{F}[1,n]\}.$
Obviously, $A(n)=\sum_{i=1}^na(i)$.

By Property \ref{P4.1}, \ref{P4.2}, \ref{P5.5}, and the recursive structure of $\mathcal{P}_F$, we have:

\begin{theorem}\label{P4.3} The vectors $[a(1)]=[1]$, $[a(2),a(3)]=[1,2]$ and for $m\geq3$
\begin{equation*}
\begin{split}
&[a(f_{m}-1),\cdots,a(f_{m+1}-2)]\\
=&[a(f_{m-2}-1),\cdots,a(f_{m-1}-2),a(f_{m-1}-1),\cdots,a(f_{m}-2)]
+[\underbrace{1,\cdots,1}_{f_{m-1}}].
\end{split}
\end{equation*}
\end{theorem}

The first few values of $a(n)$ are
$[a(1)]=[1]$,
$[a(2),a(3)]=[1,2]$,

$[a(4),a(5),a(6)]=[a(1),a(2),a(3)]+[1,1,1]=[2,2,3]$,

$[a(7),\cdots,a(11)]=[a(2),\cdots,a(6)]+[1,1,1,1,1]=[2,3,3,3,4]$,

$[a(12),\cdots,a(19)]=[a(4),\cdots,a(11)]+[\underbrace{1,\cdots,1}_8]=[3,3,4,3,4,4,4,5]$.

$[a(20),\cdots,a(32)]=[a(7),\cdots,a(19)]+[\underbrace{1,\cdots,1}_{13}]=[3,4,4,4,5,4,4,5,4,5,5,5,6]$.

By considering $A(f_{m}-2)$ for $m\geq2$, we can determine the expressions of $A(f_{m})$ etc, and give a fast algorithm of $A(n)$ for all $n\geq1$.
Let $C(m)=A(f_{m+1}-2)-A(f_{m}-2)$ for $m\geq1$. An immediate corollary of Theorem \ref{P4.3} is
$C(m)=\sum_{n=f_{m}-1}^{f_{m+1}-2}a(n)
=\sum_{n=f_{m-2}-1}^{f_{m-1}-2}a(n)+\sum_{n=f_{m-1}-1}^{f_{m}-2}a(n)+f_{m-1}.$
This means $C(m)=C(m-2)+C(m-1)+f_{m-1}$.
By induction, we have

\begin{property}\label{PCm}  $C(m)=\sum_{i=-1}^{m-1} f_if_{m-i-2}-f_{m-1}
=\frac{m+1}{5}f_{m+1}+\frac{m-2}{5}f_{m-1}$ for $m\geq1$.
\end{property}

\noindent\emph{Remark.}  We firstly prove the expression
$\sum_{i=-1}^{m}f_if_{m-i-1}=\frac{m+2}{5}f_{m+2}+\frac{m+4}{5}f_{m}$ for $m\geq-1$.

\vspace{0.2cm}

Since $A(f_{m}-2)=\sum_{n=1}^{m-1}C(n)$ and $\sum_{i=-1}^mf_i= f_{m+2}-1$, we have the expression of $A(f_{m}-2)$ as below. In the proof, the Fibonacci representation is useful.
Such a representation can be written as a string $[a_1,a_2,\ldots,a_m]_F$ representing the integer $\sum_{i=-1}^{m-2}f_ia_{m-i-1}$. For example, the string $[1,0,0,1]_F$ is the Fibonacci representation of $f_2+f_{-1}=4$. Here we don't request $a_i\in\{0,1\}$.
Using the Fibonacci representation, $C(m)=\sum_{i=-1}^{m-1} f_if_{m-i-2}-f_{m-1}=[f_{-1},f_0,\ldots,f_{m-3},f_{m-2},0]_F$.

\begin{property}\label{PAf} $A(f_{m}-2)=\sum_{i=-1}^{m}f_if_{m-i-1}-f_{m+2}-f_{m}+2=\frac{m-3}{5}f_{m+2}+\frac{m-1}{5}f_{m}+2$ for $m\geq2$.
\end{property}

\begin{proof} For $m\geq2$, since $\sum_{i=-1}^{m}f_if_{m-i-1}=\frac{m+2}{5}f_{m+2}+\frac{m+4}{5}f_{m}$,
we only need to prove
$A(f_{m}-2)=\sum_{i=-1}^{m}f_if_{m-i-1}-f_{m+2}-f_{m}+2$.
Since $\sum_{i=-1}^mf_i= f_{m+2}-1$ and using the Fibonacci representation,
$$\begin{array}{rl}
&A(f_{m}-2)=\sum_{n=1}^{m-1}C(n)=\sum_{n=1}^{m-1}[f_{-1},f_0,\ldots f_{n-3},f_{n-2},0]_F\\
=&[f_{-1},f_{-1}+f_0,\ldots,\sum_{i=-1}^{m-4}f_{i},\sum_{i=-1}^{m-3}f_{i},0]_F
=[f_{1}-1,f_{2}-1,\ldots,f_{m-2}-1,f_{m-1}-1,0]_F\\
=&[f_{1},f_{2},\ldots,f_{m-2},f_{m-1},0]_F-\sum_{i=0}^{m-2}f_{i}
=[f_{1},f_{2},\ldots,f_{m-2},f_{m-1},0]_F-f_{m}+2;\\
&\sum_{i=-1}^{m}f_if_{m-i-1}-f_{m+2}-f_{m}+2
=[f_{-1},f_0,f_1,\ldots f_{m-2},f_{m-1},f_{m}]_F-f_{m+2}-f_{m}+2\\
=&[f_1,\ldots,f_{m-2},f_{m-1},0]_F-f_{-1}f_{m+2}-f_0f_{m}+f_{-1}f_{m}+f_0f_{m-1}+f_{m}f_{-1}+2.
\end{array}$$
Since $-f_{m+2}-f_{m}+f_{m}+f_{m-1}+f_{m}=-f_{m+2}+f_{m-1}+f_{m}=-f_{m}$,
the conclusion holds.
\end{proof}

By induction and using Theorem \ref{P4.3}, we have
$a(f_{m}-2)=m-1$, $a(f_{m}-1)=\lfloor\frac{m+1}{2}\rfloor$, $a(f_{m})=\lfloor\frac{m+2}{2}\rfloor$ and $a(f_{m}-1)+a(f_{m})=m+1$.
Thus we can determine the expressions of $A(f_{m}-3)$, $A(f_{m}-1)$, $A(f_{m})$.
Especially, since $A(f_{m})=A(f_{m}-2)+a(f_{m}-1)+a(f_{m})$, we have

\begin{theorem}[]
$A(f_m)=\frac{m-3}{5}f_{m+2}+\frac{m-1}{5}f_{m}+m+3$ for $m\geq0$.
\end{theorem}

\noindent\emph{Example.} $A(f_5)=A(13)=\frac{2}{5}f_{7}+\frac{4}{5}f_{5}+8
=\frac{2}{5}\times34+\frac{4}{5}\times13+8=32$.

\vspace{0.2cm}

For any $n\geq1$, let $m$ such that $f_m\leq n+1< f_{m+1}$.
Then $A(n)=A(f_m-2)+\sum_{i=f_{m}-1}^{n}a(i)$.
In order to give a fast algorithm
of $A(n)$, we only need to calculate $\sum_{i=f_{m}-1}^{n}a(i)$. One method is calculating $a(n)$ by Theorem \ref{P4.3}, the other method is using the corollary as below.

\begin{corollary}[]\label{c5} For $n\geq1$, let $m$ such that $f_m\leq n+1< f_{m+1}$, then
\begin{equation*}
\sum_{i=f_{m}-1}^{n}a(i)=
\begin{cases}
\sum\limits_{i=f_{m-2}-1}^{n-f_{m-1}}a(i)+n-f_{m}+2,&n+1< 2f_{m-1};\\
\sum\limits_{i=f_{m-1}-1}^{n-f_{m-1}}a(i)+C(m-2)+n-f_{m}+2,&otherwise.
\end{cases}
\end{equation*}
\end{corollary}

\begin{proof} When $f_m\leq n+1< 2f_{m-1}$,
$\sum_{i=f_{m}-1}^{n}a(i)=\sum_{i=f_{m-2}-1}^{n-f_{m-1}}(a(i)+1)
=\sum_{i=f_{m-2}-1}^{n-f_{m-1}}a(i)+n-f_{m}+2$.

When $2f_{m-1}\leq n+1< f_{m+1}$, $\sum_{i=f_{m}-1}^{n}a(i)
=\sum_{i=f_{m}-1}^{2f_{m-1}-2}a(i)+\sum_{i=2f_{m-1}-1}^{n}a(i)$,
where
$$\begin{array}{rl}
\begin{cases}
\sum_{i=f_{m}-1}^{2f_{m-1}-2}a(i)=\sum_{i=f_{m-2}-1}^{f_{m-1}-2}a(i)+ f_{m-3}=C(m-2)+f_{m-3};\\
\sum_{i=2f_{m-1}-1}^{n}a(i)=\sum_{i=f_{m-1}-1}^{n-f_{m-1}}a(i)+n-2f_{m-1}+2.
\end{cases}
\end{array}$$
Thus $\sum_{i=f_{m}-1}^{n}a(i)
=\sum_{i=f_{m-1}-1}^{n-f_{m-1}}a(i)+C(m-2)+n-f_{m}+2$.
The conclusion holds.
\end{proof}

\noindent\emph{Remark.} For $2f_{m-1}\leq n+1< f_{m+1}$ in Corollary \ref{c5}, by the expression of $C(m-2)$, we can also write $\sum_{i=f_{m}-1}^{n}a(i)
=\sum_{i=f_{m-1}-1}^{n-f_{m-1}}a(i)+n+\frac{m-11}{5}f_{m-1}+\frac{m+1}{5}f_{m-3}+2$.
But we prefer the expression in Corollary \ref{c5}, in order to compare with the expression in Corollary \ref{ct} for the Tribonacci sequence.

\vspace{0.2cm}

\noindent\emph{Example.} We calculate $\sum_{i=20}^{29}A(i)$. One method is using Theorem \ref{P4.3},
$\sum_{i=20}^{29}a(i)=3+4+4+4+5+4+4+5+4+5=42$.
The other method is using Corollary \ref{c5}.
Since $f_6=21\leq 29+1< f_{7}=34$, $m=6$. Moreover $2f_5=26\leq 29+1$, and $C(4)=15$ by Property \ref{PCm},
$$\begin{array}{c}
\sum_{i=f_{6}-1}^{29}a(i)
=\sum_{i=f_{5}-1}^{29-f_{5}}a(i)+C(4)+29-f_{6}+2
=\sum_{i=f_{5}-1}^{16}a(i)+25.
\end{array}$$
Similarly, $\sum_{i=f_{5}-1}^{16}a(i)=\sum_{i=f_{4}-1}^{8}a(i)+12$ and
$\sum_{i=f_{4}-1}^{8}a(i)=a(2)+a(3)+2=5$.

Thus $\sum_{i=f_{6}-1}^{29}a(i)=25+12+5=42$.

\vspace{0.2cm}

\noindent\textbf{Algorithm A.}

Step 1. Find the $m$ such that $f_m\leq n+1< f_{m+1}$, then $A(n)=A(f_m-2)+\sum_{i=f_{m}-1}^{n}a(i)$.

Step 2. Calculate $A(f_m-2)$ by the expression in Property \ref{PAf};

Calculate $\sum_{i=f_{m}-1}^{n}a(i)$ by the recursive relation in Corollary \ref{c5} or Theorem \ref{P4.3}.

\vspace{0.2cm}

\noindent\emph{Remark.} When $n$ is large (resp. small), Corollary \ref{c5} (resp. Theorem \ref{P4.3}) is faster.

\vspace{0.2cm}

\noindent\emph{Example.} Since $f_6\leq 29+1< f_{7}$, $m=6$.
Then
$A(29)=A(19)+\sum_{i=20}^{29}a(i)$.
By Property \ref{PAf}, $A(19)=
\frac{3}{5}f_{8}+\frac{5}{5}f_{6}+2=56$.
By Theorem \ref{P4.3} or Corollary \ref{c5}, $\sum_{i=20}^{29}a(i)=42$.
Thus $A(29)=98$.

\section{Preliminaries of the Tribonacci sequence}

Now we turn to discuss the numbers of repeated palindromes in $\mathbb{T}$.
Since we discuss the properties of $\mathbb{F}$ and $\mathbb{T}$ in different sections,
we prefer to still use the notation $K_m$, $Ker(\cdot)$, $P(\omega,p)$, $\langle K_m,p\rangle$ etc.

We define $T_m=\tau^m(a)$ for $m\geq0$, $T_{-2}=\varepsilon$, $T_{-1}=c$.
Then $T_0=a$, $T_1=ab$, $T_{m}=T_{m-1}T_{m-2}T_{m-3}$ for $m\geq2$,
and $|T_m|=t_m$ for $m\geq-2$.
Denote by $\delta_m$ the last letter of $T_m$ for $m\geq-1$, then $\delta_m=a$ (resp. $b$, $c$) for $m\equiv0$, (resp. 1, 2) mod 3.
We define the kernel numbers that $k_{0}=0$, $k_{1}=k_{2}=1$, $k_m=k_{m-1}+k_{m-2}+k_{m-3}-1$ for $m\geq3$.
The kernel word with order $m$ is defined as
$K_1=a$, $K_2=b$, $K_3=c$, $K_m=\delta_{m-1}T_{m-3}[1,k_m-1]$ for $m\geq4$.
By Proposition 2.7 in \cite{HW2015-2}, all kernel words are palindromes.
Let $Ker(\omega)$ be the maximal kernel word occurring in factor $\omega$, then by Theorem 4.3 in \cite{HW2015-2}, $Ker(\omega)$ occurs in $\omega$ only once.

\begin{property}[Theorem 4.11 in \cite{HW2015-2}]\label{wpt}
$Ker(\omega_p)=Ker(\omega)_p$ for all $\omega\in\mathbb{T}$ and $p\geq1$.
\end{property}

\begin{property}[Theorem 3.3 in \cite{HW2015-2}]\label{Gt}
The gap sequence $\{G_p(\omega)\}_{p\geq1}$ is the Tribonacci sequence over the alphabet $\{G_1(\omega)$,$G_2(\omega)$,$G_4(\omega)\}$ for all $\omega\in\mathbb{T}$.
\end{property}

The next two properties are useful in our proofs, and can be proved easily by induction.

\begin{property}[]\label{bp1}~
(1) $k_m=k_{m-3}+t_{m-4}=\frac{t_{m-3}+t_{m-5}+1}{2}$ for $m\geq3$;

(2) $K_m=\delta_{m-1}T_{m-4}K_{m-3}[2,k_{m-3}]=\delta_{m-1}T_{m-4}T_{m-5}[1,k_{m-3}-1]$ for $m\geq4$.
\end{property}

\begin{property}[]\label{bp3}~(1) $\sum_{i=0}^{m}t_i=\frac{t_m+t_{m+2}-3}{2}$
for $m\geq0$, (see Lemma 6.7 in \cite{G2006});

(2) $\sum_{i=1}^{m}k_i=\frac{k_m+k_{m+2}+m-1}{2}=\frac{t_{m-2}+t_{m-3}+m}{2}$ for $m\geq1$.
\end{property}

\section{The recursive structure of $\mathcal{P}_T$}

We denote by $|\omega|_\alpha$ the number of letter $\alpha$ occurring in $\omega$.
By an analogous argument as in Section 3, we establish the recursive structures of
$\mathcal{P}_T$ in this section.
Since $P(\omega,p)=L(\omega,p)+|\omega|-1$, we can rewrite Theorem 6.1 and Remark 6.2 in \cite{HW2015-2} as

\begin{property}[]\label{Pt} For $m,p\geq1$,

$P(K_m,p)=pt_{m-1}+|\mathbb{T}[1,p-1]|_a(t_{m-2}+t_{m-3})+|\mathbb{T}[1,p-1]|_bt_{m-2}+k_m-1$.
\end{property}

\noindent\emph{Example.}
$P(K_m,1)=t_{m-1}+k_m-1=k_{m+3}-1$ for $m\geq1$.
$P(a,p)=p+|\mathbb{T}[1,p-1]|_a+|\mathbb{T}[1,p-1]|_b$,
$P(b,p)=2p+2|\mathbb{T}[1,p-1]|_a+|\mathbb{T}[1,p-1]|_b$,
$P(c,p)=4p+3|\mathbb{T}[1,p-1]|_a+2|\mathbb{T}[1,p-1]|_b$ for $p\geq1$.

\vspace{0.2cm}

By Proposition 3.2, 4.9 and Corollary 4.13 in \cite{HW2015-2}, any factor $\omega$ with kernel $K_m$ can be expressed uniquely as
$\omega=(K_mG_4(K_m))[i,t_{m-1}] K_{m}(G_4(K_m)K_m)[1,j],$
where $2\leq i\leq t_{m-1}+1$, $0\leq j\leq t_{m-1}-1$, $G_4(K_m)=T_{m-1}[k_{m},t_{m-1}-1]$, $m\geq1$.

\begin{property}[]
Any palindrome with kernel $K_m$ can be expressed uniquely as
$$T_{m-1}[i,t_{m-1}-1] K_{m} T_{m}[k_{m},k_{m+3}-i-1]=K_{m+4}[i+1,k_{m+4}-i],\eqno(4)$$
where $1\leq i\leq t_{m-1}$ and $m\geq1$.
\end{property}

\begin{property}[]\
Let $\omega$ be a palindrome with kernel $K_m$ satisfying the expression (4), then
$$P(\omega,p)=P(K_m,p)+t_{m-1}-i\text{ for }1\leq i\leq t_{m-1},m\geq1.\eqno(5)$$
\end{property}

\noindent\emph{Remark.} Obviously $\{n:\mathbb{T}[1,n]\in\mathcal{P}_T\}
=\{n:\omega\in\mathcal{P}_T,|\omega|=P(\omega,1)=n\}$.
By expressions (4) and (5), $|\omega|=k_{m+4}-2i$, $P(\omega,1)=k_{m+4}-i-1$.
Thus $|\omega|=P(\omega,1)$ if and only if $i=1\in\{1,\cdots,t_{m-1}\}.$
This means $\mathbb{T}[1,n]$ is a palindrome if and only if $n=k_{m+4}-2$ for $m\geq1$. H.Mousavi and J.Shallit gave this result by mechanical method, see Theorem 11 in \cite{MS2014}.

\vspace{0.2cm}

Define
$\langle K_m,p\rangle=\{P(\omega,p):\omega\in\mathcal{P}_T,Ker(\omega)=K_m\}$ for $m,p\geq1$,
which contains some consecutive integers.
By the expression (5),
$$\langle K_m,p\rangle =\{P(K_m,p)+t_{m-1}-i,1\leq i\leq t_{m-1}\}\\
=\{P(K_m,p),\cdots,P(K_m,p)+t_{m-1}-1\}.\eqno(6)$$
An immediately corollary is $\sharp\langle K_m,p\rangle=\sharp\{1\leq i\leq t_{m-1}\}=t_{m-1}$ for $m\geq1$.

\begin{lemma}[]\label{Lt}
(1) $|\mathbb{T}[1,P(a,p)]|_a=|\mathbb{T}[1,P(b,p)]|_b=|\mathbb{T}[1,P(c,p)]|_c=p$;

(2) $|\mathbb{T}[1,P(b,p)]|_a=P(a,p)$, $|\mathbb{T}[1,P(a,p)|_b=|\mathbb{T}[1,p-1]|_a$;

(3) $|\mathbb{T}[1,P(c,p)]|_a=P(b,p)$, $|\mathbb{T}[1,P(c,p)]|_b=P(a,p)$.
\end{lemma}

\begin{proof} (1) The conclusions hold by the definitions of $|\mathbb{T}[1,p]|_\alpha$ and $P(\alpha,p)$, $\alpha\in\{a,b,c\}$.

(2) By Property \ref{wpt}, the $p$-th occurrence of $aca=aG_2(a)a$ is equivalent to the $P(b,p)$-th occurrence of $a$.
Moreover, the $(P(b,p)+1)$-th letter $a$ occurs at position $P(aca,p)$.
So $P(aca,p)=P(a,P(b,p)+1)$.
Since $Ker(aca)=c$, by expression (5),
$$P(aca,p)=P(c,p)+1=4p+3|\mathbb{T}[1,p-1]|_a+2|\mathbb{T}[1,p-1]|_b+1.$$

On the other hand, $P(b,p)=2p+2|\mathbb{T}[1,p-1]|_a+|\mathbb{T}[1,p-1]|_b$, so
\begin{equation*}
\begin{split}
&P(a,P(b,p)+1)=P(b,p)+1+|\mathbb{T}[1,P(b,p)]|_a+|\mathbb{T}[1,P(b,p)]|_b\\
=&3p+2|\mathbb{T}[1,p-1]|_a+|\mathbb{T}[1,p-1]|_b+|\mathbb{T}[1,P(b,p)]|_a+1.
\end{split}
\end{equation*}

By $P(aca,p)=P(a,P(b,p)+1)$, $|\mathbb{T}[1,P(b,p)]|_a=p+|\mathbb{T}[1,p-1]|_a+|\mathbb{T}[1,p-1]|_b=P(a,p)$.

Similarly, since $aba=aG_1(a)a$, the $(P(a,p)+1)$-th letter $a$ occurs at position $P(aba,p)$. This means $P(aba,p)=P(a,P(a,p)+1)$.
By expression (5) and Property \ref{Pt}, the second equation holds.

(3) By analogous arguments, we have $P(aa,p)=P(a,P(c,p)+1)$ and $P(bab,p)=P(b,P(c,p)+1)$.
Since $aa=K_4$ and $bab=K_5$, by Property \ref{Pt}, we get two functions as below:
\begin{equation*}
\begin{cases}
|\mathbb{T}[1,P(c,p)]|_a+|\mathbb{T}[1,P(c,p)]|_b
=3p+3|\mathbb{T}[1,p-1]|_a+2|\mathbb{T}[1,p-1]|_b;\\
2|\mathbb{T}[1,P(c,p)]|_a+|\mathbb{T}[1,P(c,p)]|_b=5p+5|\mathbb{T}[1,p-1]|_a+3|\mathbb{T}[1,p-1]|_b.
\end{cases}
\end{equation*}
Thus $|\mathbb{T}[1,P(c,p)]|_a=P(b,p)$ and $|\mathbb{T}[1,P(c,p)]|_b=P(a,p)$. The conclusions hold.
\end{proof}

\begin{property}\label{p1}
$\langle K_m,p\rangle=\langle K_{m-3},P(c,p)+1\rangle\cup\langle K_{m-2},P(b,p)+1\rangle\cup\langle K_{m-1},P(a,p)+1\rangle$ for $m\geq4$.
\end{property}

\begin{property}\label{p2}
$\max\langle c,p\rangle=\max\langle b,P(a,p)+1\rangle$,
$\min\langle b,P(a,p)+1\rangle=\max\langle a,P(b,p)+1\rangle+1$,\\
$\min\langle c,p\rangle+1=\min\langle a,P(b,p)+1\rangle$,
$\max\langle b,p\rangle=\max\langle a,P(a,p)+1\rangle$.
\end{property}

By Property \ref{p1} and \ref{p2}, we establish the following recursive relations for any $\langle K_m,p\rangle$. We call it the recursive structure of $\mathcal{P}_T$.
\begin{equation*}
\begin{cases}
\pi\langle K_m,p\rangle=\langle K_{m-3},P(c,p)+1\rangle\cup\langle K_{m-2},P(b,p)+1\rangle\cup\langle K_{m-1},P(a,p)+1\rangle&\text{for }m\geq4;\\
\pi\langle K_3,p\rangle=\pi\langle c,p\rangle=\langle a,P(b,p)+1\rangle\cup\langle b,P(a,p)+1\rangle;\\
\pi\langle K_2,p\rangle=\pi\langle b,p\rangle=\langle a,P(a,p)+1\rangle.
\end{cases}
\end{equation*}

On the other hand, (a) For any $m\geq1$, each $\langle K_{m},1\rangle$ belongs to the recursive structure.
(b) Since $\mathbb{T}$ over alphabet $\{a,b,c\}$, $\mathbb{N}=\{1\}\cup\{P(a,\hat{p})+1,\hat{p}\geq1\}\cup\{P(b,\hat{p})+1,\hat{p}\geq1\}\cup\{P(c,\hat{p})+1,\hat{p}\geq1\}$. So for any $p\geq2$, there exists
$\hat{p}$ such that $P(a,\hat{p})+1=p$ or $P(b,\hat{p})+1=p$ or $P(c,\hat{p})+1=p$.
\begin{equation*}
\begin{cases}
\text{If there exists }\hat{p}\text{ such that }P(a,\hat{p})+1=p,
\begin{cases}
\langle K_{m},p\rangle\subset\pi\langle K_{m+1},\hat{p}\rangle\text{ for }m\geq3\\
\langle K_{m},p\rangle=\langle b,p\rangle\subset\pi\langle c,\hat{p}\rangle\text{ for }m=2\\
\langle K_{m},p\rangle=\langle a,p\rangle=\pi\langle b,\hat{p}\rangle\text{ for }m=1
\end{cases}\\
\text{If there exists }\hat{p}\text{ such that }P(b,\hat{p})+1=p,
\begin{cases}
\langle K_{m},p\rangle\subset\pi\langle K_{m+2},\hat{p}\rangle\text{ for }m\geq2\\
\langle K_{m},p\rangle=\langle a,p\rangle\subset\pi\langle c,\hat{p}\rangle\text{ for }m=1
\end{cases}\\
\text{If there exists }\hat{p}\text{ such that }P(c,\hat{p})+1=p,
\langle K_{m},p\rangle\subset\pi\langle K_{m+3},\hat{p}\rangle\text{ for }m\geq1.
\end{cases}
\end{equation*}
Thus the recursive structure contains all $\langle K_m,p\rangle$, i.e. contains all palindromes in $\mathbb{T}$.
\vspace{-0.3cm}
\small
\setlength{\unitlength}{1.1mm}
\begin{center}
\begin{picture}(115,62)
\put(5,1){27}
\put(30,1){27}
\put(55,1){27}
\put(80,1){27}
\put(105,1){27}
\put(101.5,7){$\langle K_1,15\rangle$}
\put(5,4){26}
\put(30,4){26}
\put(55,4){26}
\put(80,4){26}
\put(77,7){$\langle K_2,8\rangle$}
\put(5,10){25}
\put(30,10){25}
\put(55,10){25}
\put(105,10){25}
\put(101.5,16){$\langle K_1,14\rangle$}
\put(5,13){24}
\put(30,13){24}
\put(55,13){24}
\put(52,16){$\langle K_3,4\rangle$}
\put(5,19){23}
\put(30,19){23}
\put(80,19){23}
\put(105,19){23}
\put(101.5,25){$\langle K_1,13\rangle$}
\put(5,22){22}
\put(30,22){22}
\put(80,22){22}
\put(77,25){$\langle K_2,7\rangle$}
\put(5,28){21}
\put(30,28){21}
\put(105,28){21}
\put(27,31){$\langle K_4,2\rangle$}
\put(101.5,31){$\langle K_1,12\rangle$}
\put(5,34){20}
\put(55,34){20}
\put(80,34){20}
\put(105,34){20}
\put(101.5,40){$\langle K_1,11\rangle$}
\put(5,37){19}
\put(55,37){19}
\put(80,37){19}
\put(77,40){$\langle K_2,6\rangle$}
\put(5,43){18}
\put(55,43){18}
\put(105,43){18}
\put(101.5,49){$\langle K_1,10\rangle$}
\put(5,46){17}
\put(55,46){17}
\put(52,49){$\langle K_3,3\rangle$}
\put(5,52){16}
\put(80,52){16}
\put(105,52){16}
\put(102,58){$\langle K_1,9\rangle$}
\put(5,55){15}
\put(80,55){15}
\put(2,58){$\langle K_5,1\rangle$}
\put(77,58){$\langle K_2,5\rangle$}
\put(1,0){\line(1,0){11}}
\put(1,61){\line(1,0){11}}
\put(1,0){\line(0,1){61}}
\put(12,0){\line(0,1){61}}
\put(26,0){\line(1,0){11}}
\put(26,34){\line(1,0){11}}
\put(26,0){\line(0,1){34}}
\put(37,0){\line(0,1){34}}
\put(51,0){\line(1,0){11}}
\put(51,19){\line(1,0){11}}
\put(51,0){\line(0,1){19}}
\put(62,0){\line(0,1){19}}
\put(51,33){\line(1,0){11}}
\put(51,52){\line(1,0){11}}
\put(51,33){\line(0,1){19}}
\put(62,33){\line(0,1){19}}
\put(76,0){\line(1,0){11}}
\put(76,10){\line(1,0){11}}
\put(76,0){\line(0,1){10}}
\put(87,0){\line(0,1){10}}
\put(76,18){\line(1,0){11}}
\put(76,28){\line(1,0){11}}
\put(76,18){\line(0,1){10}}
\put(87,18){\line(0,1){10}}
\put(76,33){\line(1,0){11}}
\put(76,43){\line(1,0){11}}
\put(76,33){\line(0,1){10}}
\put(87,33){\line(0,1){10}}
\put(76,51){\line(1,0){11}}
\put(76,61){\line(1,0){11}}
\put(76,51){\line(0,1){10}}
\put(87,51){\line(0,1){10}}
\put(101,0){\line(0,1){61}}
\put(113,0){\line(0,1){61}}
\put(101,0){\line(1,0){12}}
\put(101,9.5){\line(1,0){12}}
\put(101,18.5){\line(1,0){12}}
\put(101,27.5){\line(1,0){12}}
\put(101,33.5){\line(1,0){12}}
\put(101,42.5){\line(1,0){12}}
\put(101,51.5){\line(1,0){12}}
\put(101,61){\line(1,0){12}}
\put(13,55){\vector(1,0){62}}
\put(13,55){\vector(3,-1){37}}
\put(13,55){\vector(1,-3){12}}
\put(38,30){\vector(1,0){62}}
\put(38,30){\vector(4,-1){37}}
\put(38,30){\vector(1,-2){12}}
\put(63,45){\vector(1,0){37}}
\put(63,45){\vector(2,-1){12}}
\put(63,12){\vector(1,0){37}}
\put(63,12){\vector(2,-1){12}}
\put(88,5){\vector(1,0){12}}
\put(88,21){\vector(1,0){12}}
\put(88,38){\vector(1,0){12}}
\put(88,55){\vector(1,0){12}}
\end{picture}
\end{center}
\normalsize
\vspace{-0.2cm}
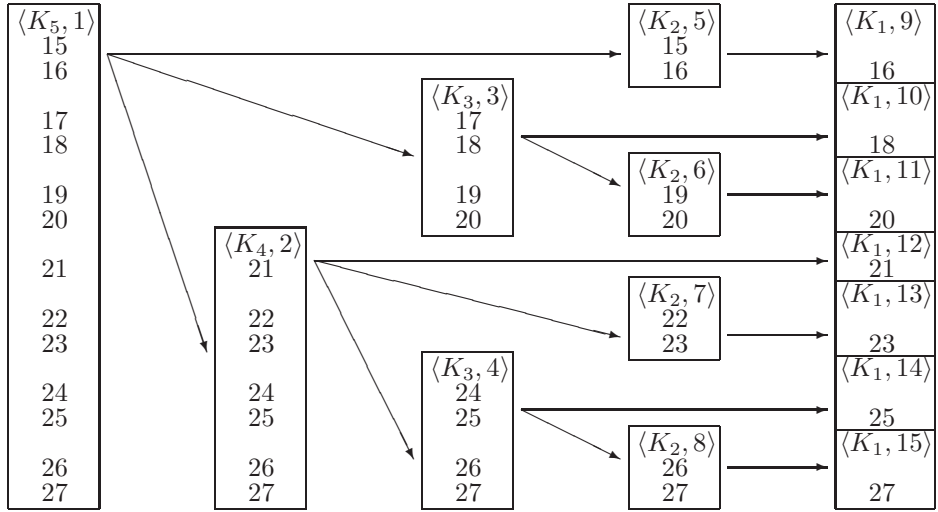
\centerline{Fig.2: The recursive structure of $\mathcal{P}_T$ from $\langle K_5,1\rangle$.}

\vspace{0.2cm}

By the recursive structure, we have the relation between the number of palindromes ending at position $P(K_m,p)+i-1$ and $k_{m+3}+i-2$, $1\leq i\leq t_{m-1}$.
They are the $i$-th element in $\langle K_m,p\rangle$ and $\langle K_m,1\rangle$, respectively.

\begin{property}[]\label{p4} For $1\leq i\leq t_{m-1}$, $m\geq-1$, $p\geq1$,
\begin{equation*}
\begin{split}
&\{\omega:\omega\in\mathcal{P}_T,\omega\triangleright\mathbb{T}[1,k_{m+3}+i-2]\}\\
=&\{\omega:\omega\in\mathcal{P}_T,\omega\triangleright\mathbb{T}[1,
P(K_m,p)+i-1],Ker(\omega)=K_j,1\leq j\leq m\}.
\end{split}
\end{equation*}
\end{property}

\noindent\emph{Example.} Taking $m=4$, $p=3$, $i=2$. All palindromes ending at position $9$ are $\{b,baab\}$.
All palindromes ending at position $33$ are $\{b,baab,\omega\}$ where $\omega=baabacabacabaab$.
Since $Ker(b)=K_{2}$, $Ker(baab)=K_4$ and $Ker(\omega)=K_6$,
only $\{b,baab\}$ are palindromes with kernel $K_{j}$, $1\leq j\leq 4$.

\section{The algorithm for counting B(n)}

By an analogous argument as in Section 4, we give algorithms for
counting $B(n)$ in this section.
Comparing with the work on $\mathbb{F}$, the only difficulty on $\mathbb{T}$ is getting explicit expression for $B(t_m)$.

Denote $b(n)=\sharp\{(\omega,p):\omega\in\mathcal{P}_T, \omega_p\triangleright\mathbb{T}[1,n]\}.$
Obviously, $B(n)=\sum_{i=1}^nb(i)$. By Property \ref{p1}, \ref{p2}, \ref{p4}, and the recursive structure of $\mathcal{P}_T$, we have:

\begin{theorem}\label{b} $[b(1)]=[1]$, $[b(2),b(3)]=[1,2]$, $[b(4),b(5),b(6),b(7)]=[1,2,2,3]$ and for $m\geq4$
\begin{equation*}
\begin{split}
&[b(k_{m+3}-1),\cdots,b(k_{m+4}-2)]\\
=&[b(k_{m}-1),\cdots,b(k_{m+1}-2),b(k_{m+1}-1),\cdots,b(k_{m+2}-2),b(k_{m+2}-1),\cdots,b(k_{m+3}-2)]\\
&~~~~~~~~~~~~~~~~~~~~~~~~~~~~~~~~~~~~~~~~~~~~~~~~~~~~~~~~~~~~~~~~~~~~~~~~~~~~~~~~~~~~~~~~~~~~~
+[\underbrace{1,\cdots,1}_{t_{m-1}}].
\end{split}
\end{equation*}
\end{theorem}

The first few values of $b(n)$ are
$[b(1)]=[1]$, $[b(2),b(3)]=[1,2]$, $[b(4),b(5),b(6),b(7)]=[1,2,2,3]$,

$[b(8),\cdots,b(14)]=[b(1),\cdots,b(7)]+[1,1,1,1,1,1,1]=[2,2,3,2,3,3,4]$,

$[b(15),\cdots,b(27)]=[b(2),\cdots,b(14)]+[\underbrace{1,\cdots,1}_{13}]
=[2,3,2,3,3,4,3,3,4,3,4,4,5]$.

By considering $B(k_{m+4}-2)$ for $m\geq1$, we can determine the expressions of $B(t_{m})$ etc, and give a fast algorithm of $B(n)$ for all $n\geq1$.
Let $D(m)=B(k_{m+4}-2)-B(k_{m+3}-1)$, then an immediate corollary of Theorem \ref{b} is
$$\begin{array}{c}
D(m)=\sum_{n=k_{m+3}-1}^{k_{m+4}-2}b(n)
=\sum_{n=k_{m}-1}^{k_{m+1}-2}b(n)+\sum_{n=k_{m+1}-1}^{k_{m+2}-2}b(n)
+\sum_{n=k_{m+2}-1}^{k_{m+3}-2}b(n)+t_{m-1}.
\end{array}$$
This means $D(m)=D(m-3)+D(m-2)+D(m-1)+t_{m-1}$. By induction, we have

\begin{property}\ For $m\geq1$,
$$\begin{array}{c}
D(m)=\sum_{i=-1}^{m-2}t_it_{m-i-2}
=\frac{m}{22}(3t_{m}+7t_{m-1}+2t_{m-2})+\frac{1}{22}(3t_{m}-3t_{m-1}+4t_{m-2}).
\end{array}$$
\end{property}

\begin{lemma}[]
$\sum_{i=0}^{m}(i+1)t_i=\frac{m}{2}(t_{m+2}+t_{m})-\frac{1}{2}(t_{m-1}+t_{m-2})+\frac{3}{2}$ for $m\geq0$.
\end{lemma}

\begin{property}\label{B}\ $B(k_{m+4}-2)=\frac{m}{44}(8t_{m+2}+4t_{m+1}-2t_{m})
-\frac{1}{44}(9t_{m+2}-3t_{m})+\frac{3}{4}$ for $m\geq1$.
\end{property}

The three properties above can be proved easily by induction. For any $n\geq1$, let $m$ such that $k_{m+4}\leq n+1< k_{m+5}$.
Then $B(n)=B(k_{m+4}-2)+\sum_{i=k_{m+4}-1}^nb(i)$. In order to give a fast algorithm of $B(n)$, we only need to calculate $\sum_{i=k_{m+4}-1}^nb(i)$. One method is calculating $b(n)$ by Theorem \ref{b}, the other method is using the corollary as below.

\begin{corollary}\label{ct} For $n\geq15$, let $m$ such that $k_{m+3}\leq n+1< k_{m+4}$, then $m\geq4$
\begin{equation*}
\sum_{i=k_{m+3}-1}^nb(i)
=\begin{cases}
\sum\limits_{i=k_{m}-1}^{n-t_{m-1}}b(i)+n-k_{m+3}+2,&n+1< \alpha_m;\\
\sum\limits_{i=k_{m+1}-1}^{n-t_{m-1}}b(i)+D(m-3)+n-k_{m+3}+2,
&\alpha_m\leq n+1< \beta_m;\\
\sum\limits_{i=k_{m+2}-1}^{n-t_{m-1}}b(i)+D(m-3)+D(m-2)+n-k_{m+3}+2,
&\beta_m\leq n+1.
\end{cases}
\end{equation*}
where $D(m)=\frac{m}{22}(3t_{m}+7t_{m-1}+2t_{m-2})+\frac{1}{22}(3t_{m}-3t_{m-1}+4t_{m-2})$, $\alpha_m=k_{m+3}+t_{m-4}$ and $\beta_m=k_{m+3}+t_{m-4}+t_{m-3}$.
\end{corollary}

\begin{proof} 
When $n+1< \alpha_m$,
$\sum_{i=k_{m+3}-1}^nb(i)=\sum_{i=k_{m}-1}^{n-t_{m-1}}[b(i)+1]
=\sum_{i=k_{m}-1}^{n-t_{m-1}}b(i)+n-k_{m+3}+2.$

When $\alpha_m\leq n+1< \beta_m$, since $k_{m+3}-t_{m-1}=k_m$ and $k_{m}+t_{m-4}=k_{m+1}$,
$$\begin{array}{rl}
&\sum\limits_{i=k_{m+3}-1}^nb(i)
=\sum\limits_{i=k_{m+3}-1}^{k_{m+3}+t_{m-4}-2}b(i)+\sum\limits_{i=k_{m+3}+t_{m-4}-1}^nb(i)\\
=&\sum\limits_{i=k_{m}-1}^{k_{m+1}-2}[b(i)+1]
+\sum\limits_{i=k_{m+1}-1}^{n-t_{m-1}}[b(i)+1]
=\sum\limits_{i=k_{m+1}-1}^{n-t_{m-1}}b(i)+D(m-3)+n-k_{m+3}+2.
\end{array}$$

When $\beta_m\leq n+1$, since $k_{m+3}-t_{m-1}=k_m$, $k_{m}+t_{m-4}=k_{m+1}$ and
$k_{m+1}+t_{m-3}=k_{m+2}$,
$$\begin{array}{rl}
&\sum\limits_{i=k_{m+3}-1}^nb(i)
=\sum\limits_{i=k_{m+3}-1}^{k_{m+3}+t_{m-4}-2}b(i)
+\sum\limits_{i=k_{m+3}+t_{m-4}-1}^{k_{m+3}+t_{m-4}+t_{m-3}-2}b(i)
+\sum\limits_{i=k_{m+3}+t_{m-4}+t_{m-3}-1}^nb(i)\\
=&\sum\limits_{i=k_{m}-1}^{k_{m+1}-2}b(i)
+\sum\limits_{i=k_{m+1}-1}^{k_{m+2}-2}b(i)
+\sum\limits_{i=k_{m+2}-1}^{n-t_{m-1}}b(i)+n-k_{m+3}+2\\
=&\sum\limits_{i=k_{m+2}-1}^{n-t_{m-1}}b(i)+D(m-3)+D(m-2)+n-k_{m+3}+2.
\end{array}$$
Thus the conclusion holds.
\end{proof}

\noindent\emph{Example.} We calculate $\sum_{i=15}^{24}b(i)$.
One method is using Theorem \ref{b},
Since $[b(15),\cdots,b(27)]=[2,3,2,3,3,4,3,3,4,3,4,4,5]$, $\sum_{i=15}^{24}b(i)=30$.
The other method is using Corollary \ref{ct}.
Since $k_{8}\leq n+1< k_{9}$ for $n=24$, $m=5$.
Moreover $\beta_5=k_{8}+t_{1}+t_{2}=22\leq 24+1$, then
$$\begin{array}{rl}
\sum_{i=15}^{24}b(i)
=&\sum_{i=k_{7}-1}^{24-t_{4}}b(i)
+D(2)+D(3)+n-k_{8}+2
=\sum_{i=8}^{11}b(i)
+21=30.
\end{array}$$

\noindent\textbf{Algorithm B.}

Step 1. Find the $m$ such that $k_{m+3}\leq n+1< k_{m+4}$.

Step 2. Calculate $B(k_{m+3}-2)$ by Property \ref{B}; calculate $\sum_{i=k_{m+3}-1}^nb(i)$ by Theorem \ref{b} or by Corollary \ref{ct}.
Then $B(n)=B(k_{m+3}-2)+\sum_{i=k_{m+3}-1}^nb(i)$.

\vspace{0.2cm}

\noindent\emph{Example.} We calculate $B(24)$. Since $k_{8}\leq 24+1< k_{9}$, $m=5$.
By Property \ref{B},
$$\begin{array}{c}
B(k_{8}-2)=B(14)
=\frac{4}{44}(8t_{6}+4t_{5}-2t_{4})
-\frac{1}{44}(9t_{6}-3t_{4})+\frac{3}{4}=31.
\end{array}$$
By Theorem \ref{b} or by Corollary \ref{ct}, $\sum_{i=15}^{24}b(i)=30$.
Thus $B(24)=B(14)+\sum_{i=15}^{24}b(i)=61$.

\vspace{0.2cm}

Now we turn to get explicit expression for $B(t_m)$.

For $m\geq4$. (1) Since $k_{m+3}\leq t_m+1< k_{m+4}$ and $\beta_m=k_{m+3}+t_{m-4}+t_{m-3}\leq t_m+1$,
$$\begin{array}{c}
\sum_{i=k_{m+3}-1}^{t_m}b(i)=\sum_{i=k_{m+2}-1}^{t_m-t_{m-1}}b(i)+D(m-3)+D(m-2)+t_m-k_{m+3}+2.
\end{array}$$
(2) Since $k_{m+2}\leq t_m-t_{m-1}+1< k_{m+3}$ and $\alpha_{m-1}\leq t_m-t_{m-1}+1\leq \beta_{m-1}$,
$$\begin{array}{c}\sum_{i=k_{m+2}-1}^{t_m-t_{m-1}}b(i)
=\sum_{i=k_{m}-1}^{t_m-t_{m-1}-t_{m-2}}b(i)+D(m-4)+t_m-t_{m-1}-k_{m+2}+2.
\end{array}$$

By (1) and (2) above, using $D(m-1)=D(m-4)+D(m-3)+D(m-2)+t_{m-2}$, we have
$$\begin{array}{c}
\sum_{i=k_{m+3}-1}^{t_m}b(i)
=\sum_{i=k_{m}-1}^{t_{m-3}}b(i)+D(m-1)+t_m+t_{m-3}-k_{m+2}-k_{m+3}+4.
\end{array}$$

Since $B(t_m)=\sum_{n=1}^{m-1}D(n)+\sum_{i=k_{m+3}-1}^{t_m}b(i)$
and $B(t_{m-3})=\sum_{n=1}^{m-4}D(n)+\sum_{i=k_{m}-1}^{t_{m-3}}b(i)$,
$$\begin{array}{rl}
&B(t_m)-B(t_{m-3})\\
=&D(m-1)+D(m-2)+D(m-3)+D(m-1)+t_m+t_{m-3}-k_{m+2}-k_{m+3}+4\\
=&D(m)+D(m-1)+t_{m-2}+2t_{m-3}-k_{m+2}-k_{m+3}+4.
\end{array}$$
Using $t_{m-2}-k_{m+3}=-k_{m+2}$, $t_{m-3}-k_{m+2}=-k_{m+1}$, and checking for $m=3$, we have

\begin{property}[]\
$B(t_m)=B(t_{m-3})+D(m)+D(m-1)-2k_{m+1}+4$ for $m\geq3$.

\end{property}

Using the expression of $D(m)$, $k_m=\frac{t_{m-3}+t_{m-5}+1}{2}$ and by induction, we have

\begin{theorem}[]\
$B(t_m)=\tfrac{m}{22}(10t_{m}+5t_{m-1}+3t_{m-2})
+\tfrac{1}{22}(-23t_{m}+12t_{m-1}-5t_{m-2})+m+\frac{3}{2}$ for $m\geq0$.
\end{theorem}

\noindent\emph{Example.} $B(24)=B(t_5)=\tfrac{5}{22}(10t_{5}+5t_{4}+3t_{3})
+\tfrac{1}{22}(-23t_{5}+12t_{4}-5t_{3})+5+\frac{3}{2}=61$.

\section{The cylinder structure and chain structure}

As a final remark, we establish the cylinder structures and chain structures of $\mathcal{P}_F$ and $\mathcal{P}_T$ in this section. Using them, we prove some known results.

\vspace{0.4cm}

\noindent\textbf{\large{8.1. The cylinder structure}}

\vspace{0.2cm}

By the expression (1) in Section 3, we see that the set $\mathcal{P}_F$ decomposes into three disjoint cylinder sets
$\langle a\rangle$, $\langle b\rangle$ and $\langle aa\rangle$, which can be illustrated by the following way, see Tab.1. We call it the cylinder structure of $\mathcal{P}_F$.
This means any palindromes can be generated by kernel words.
More concretely, let $\omega$ be a palindrome, then $\omega\in \langle aa\rangle$ if $|\omega|$ is even;
$\omega\in\langle a\rangle$ (resp. $\langle b\rangle$) if $|\omega|$ is odd and the middle letter of $\omega$ is $a$ (resp. $b$).

Tab.1 shows the first several elements of cylinder sets $\langle a\rangle$, $\langle b\rangle$ and $\langle aa\rangle$, where we sign all kernel words with underlines.
We can see that the kernel words are sparse in $\mathcal{P}_F$, this means the palindromes are generated ``high efficient'' from kernel words.

\vspace{0.2cm}

\centerline{Tab.1: The cylinder structure of $\mathcal{P}_F$.}
\vspace{-0.2cm}
\begin{center}
\begin{tabular}{ccc}
\hline
cylinder$\langle a\rangle$&cylinder$\langle b\rangle$&cylinder$\langle aa\rangle$\\\hline
\underline{a}&\underline{b}&\underline{aa}\\
\underline{bab}&aba&baab\\
ababa&\underline{aabaa}&abaaba\\
aababaa&baabaab&\underline{babaabab}\\
baababaab&abaabaaba&ababaababa\\
abaababaaba&babaabaabab&aababaababaa\\
\underline{aabaababaabaa}&ababaabaababa&baababaababaab\\
baabaababaabaab&aababaabaababaa&abaababaababaaba\\
abaabaababaabaaba&baababaabaababaab&aabaababaababaabaa\\
babaabaababaabaabab&abaababaabaababaaba&baabaababaababaabaab\\
ababaabaababaabaababa&\underline{babaababaabaababaabab}&abaabaababaababaabaaba\\
aababaabaababaabaababaa&ababaababaabaababaababa&babaabaababaababaabaabab\\
\vdots&\vdots&\vdots\\\hline
\end{tabular}
\end{center}

By the cylinder structure of $\mathcal{P}_F$, we get some known results immediately.
Let $\mathcal{P}_F(n)$ be all palindromes occurring in $\mathbb{F}$ of length $n$.
For a finite word $\omega=x_1x_2\cdots x_n$, the $i$-th conjugation of $\omega$ is the word $C_i(\omega)=x_{i+1}\cdots x_nx_1\cdots x_i$ where $0\leq i\leq n-1$.

(1) For any $n\geq1$, $\sharp\mathcal{P}_F(n)=2$ if $n$ is odd; $\sharp\mathcal{P}_F(n)=1$ if $n$ is even. See X.Droubay\cite{D1995}.

(2) For any $m\geq-1$, $\mathcal{P}_F(f_m)\cap\{C_i(F_m),0\leq i\leq f_{m}\}=0$ ($m\equiv1\mod~3$) or 1 (otherwise). See both W-.F.Chuan\cite{C1993} and X.Droubay\cite{D1995}.

(3) For any $m\geq-1$, $K_{m}\not\!\prec K_{m+1}$. See Z.-X.Wen and Z.-Y.Wen\cite{WW1994}.

Similarly, we establish the cylinder structure of $\mathcal{P}_T$.
The set $\mathcal{P}_T$ decomposes into four disjoint cylinder sets
$\langle a\rangle$, $\langle b\rangle$, $\langle c\rangle$ and $\langle aa\rangle$.
By the cylinder structure of $\mathcal{P}_T$, we have

\begin{property}[]\
For $m\geq3$, $K_{m-1}\not\!\prec K_{m}$, $K_{m-2}\not\!\prec K_{m}$
and $K_{m-3}\prec K_{m}$ in $\mathbb{T}$.
\end{property}

\vspace{0.2cm}

\noindent\textbf{\large{8.2. The chain structure}}

\vspace{0.2cm}

As a special case of the expression (3) in Section 3, $\langle K_m,1\rangle=\{f_{m+2}-1,\cdots,f_{m+3}-2\}.$
Thus for $m\geq -1$, two integer sets $\langle K_m,1\rangle$ and $\langle K_{m+1},1\rangle$ are consecutive.
Therefore we get a chain $\{\langle K_m,1\rangle\}_{m\geq-1}$ satisfying $\bigcup_{m=-1}^\infty\langle K_m,1\rangle=\mathbb{N}$ below, which we called the chain structure of $\mathcal{P}_F$.

\vspace{0.2cm}

\centerline{Tab.2: The chain structure of $\mathcal{P}_F$.}
\vspace{-0.2cm}
\small
\begin{center}
\begin{tabular}{l|l|l|l|l|l|l|l|l}%
\hline
$\langle a,1\rangle$&$\langle b,1\rangle$&$\langle aa,1\rangle$&$\langle K_2,1\rangle$&$\langle K_3,1\rangle$&$\langle K_4,1\rangle$&$\langle K_5,1\rangle$&$\langle K_6,1\rangle$&$\ldots$\\
$\{1\}$&$\{2,3\}$&$\{4,5,6\}$&$\{7,\cdots,11\}$&$\{12,\cdots,19\}$
&$\{20,\cdots,32\}$&$\{33,\cdots,53\}$&$\{54,\cdots,87\}$&\\\hline
\end{tabular}
\end{center}
\normalsize

Similarly, we establish the chain structure of $\mathcal{P}_T$.
By the chain structure of $\mathcal{P}_F$ (resp. $\mathcal{P}_T$), the number of distinct palindromes in $\mathbb{F}[1,n]$ (resp. $\mathbb{T}[1,n]$) is $n$ (without empty word) for $n\geq1$.
Since $|\mathbb{F}[1,n]|=n$, this results are consistent with ``both $\mathbb{F}$ and $\mathbb{T}$ are rich words".

\vspace{0.5cm}

\noindent\textbf{\Large{Acknowledgments}}

\vspace{0.4cm}

The research is supported by the Grant NSFC No.11431007, No.11271223 and No.11371210.

\end{CJK*}
\end{document}